\documentclass[runningheads]{llncs}
\usepackage{enumerate}
\usepackage{ifthen,xspace}
\usepackage{graphicx,color,stmaryrd}
\usepackage{amsmath,amsfonts,amstext,amssymb,mathabx}
\makeatletter\@ifclassloaded{llncs}{}{\usepackage{amsthm}}\makeatother

\newtheorem{prop}{Proposition}

\newtheorem{obs}[prop]{Observation}
\newtheorem{thm}[prop]{Theorem}
\newtheorem{qstn}[prop]{Question}

\newcommand{\FE}{\sqsubseteq}
\newcommand{\FEbar}{\sqsupseteq}

\newcommand{\ZFC}{{\rm ZFC}}

\newcommand{\GCH}{{\rm GCH}}
\newcommand{\GA}{{\rm GA}}

\newcommand{\Add}{\mathop{\rm Add}}
\newcommand{\Coll}{\mathop{\rm Coll}}

\newcommand{\Boxup}{\ensuremath{\boxtriangleup}}
\newcommand{\Boxdown}{\rotatebox[origin=c]{180}{\ensuremath{\boxtriangleup}}}
\newcommand{\subBoxup}{{\ensuremath{\boxtriangleup}}}
\newcommand{\subBoxdown}{{\scriptsize\rotatebox[origin=c]{180}{\ensuremath{\boxtriangleup}}}}

\renewcommand{\P}{\mathbb{P}}
\newcommand{\bfL}{\mathbf{L}}

\newcommand{\restrict}{{\upharpoonright}} % uses amssymb

\newcommand{\elesub}{\prec}

\newcommand{\df}{\it} % use italic for definition terms. Idea: also use this to create an index of definitions, if MakeIndex is true.

\newcommand{\axiomf}[1]{{\rm #1}}
\newcommand{\theoryf}[1]{\hbox{$\mathsf{#1}$}}

\newcommand{\MLF}[1]{\mathrm{ML}(\Boxup,#1)}
\newcommand{\MLG}[1]{\mathrm{ML}(\Boxdown,#1)}

\newenvironment{proofsketch}{\emph{Proof sketch.}}{\hfill$\Box$\\ \smallskip}

\newcommand{\wech}[1]{}

\begin{document}
\authorrunning{J.\ D.\ Hamkins, B.\ L\"owe}
\author{Joel David Hamkins$^{1,2}$, Benedikt L\"owe$^{3,4}$\thanks{The research of the first author has been supported in part by NSF grant DMS-0800762, PSC-CUNY grant 64732-00-42 and Simons Foundation grant 209252. The second author would like to thank the CUNY Graduate Center in New York for their hospitality during his sabbatical in the fall of 2009. In addition, both authors acknowledge the generous support provided to them as visiting fellows in Spring 2012 at the Isaac Newton Institute for Mathematical Sciences in Cambridge, U.K.}}
\title{Moving up and down in the generic multiverse}

\institute{Mathematics Program, The Graduate Center of The City University of New York, 365 Fifth Avenue, New York, NY 10016, U.S.A.; \texttt{jhamkins@gc.cuny.edu} \and Department of Mathematics, The College of Staten Island of CUNY, Staten Island, NY 10314, U.S.A.,
\and
Institute for Logic, Language
and Computation, Universiteit van Amsterdam, Postbus 94242, 1090 GE
Amsterdam, The Netherlands; \texttt{b.loewe@uva.nl} \and
Department Mathematik, Universit\"at Hamburg, Bundesstrasse
55, 20146 Hamburg, Germany}

\maketitle

\begin{abstract} We investigate the \emph{modal logic of the generic multiverse} which is a bimodal logic with operators corresponding to the relations ``is a forcing extension of'' and ``is a ground model of''. The fragment of the first relation is the \emph{modal logic of forcing} and was studied by the authors in earlier work. The fragment of the second relation is the \emph{modal logic of grounds} and will be studied here for the first time. In addition, we discuss which combinations of modal logics are possible for the two fragments.
\end{abstract}

\section{Introduction}

\subsection{The generic multiverse}\label{ssec:multiverse}

Recently, the \emph{generic multiverse} has become a concept of great interest to set theorists and philosophers of mathematics alike:

\begin{quote}
The generic multiverse is generated from each universe
of the collection by closing under generic extensions (enlargements) and under generic
refinements (inner models of a universe which the given universe is a generic extension
of). (W.\ Hugh Woodin, \cite{woodinpreprint})\end{quote}

If you fix a universe of set theory $V$ and iteratively build all set forcing extensions and ground models---and throughout this article unless stated specifically otherwise we shall mean always set forcing when discussing forcing---you naturally produce a Kripke structure $\mathrm{GM}(V)$ consisting of these worlds together with the accessibility relation $M \FE N$ (``$N$ is a forcing extension of $M$'') and its converse relation $N \FEbar M$ (``$M$ is a ground of $N$''). Hugh Woodin has investigated \emph{multiverse truth}, that is, truth in all models of the generic multiverse, in connection with his programme to solve the alethic status of the continuum hypothesis.  The first author has proposed a philosophy of mathematics based on a broader multiverse perspective, where we have many
different legitimate concepts of set (not merely those arising by set forcing), and these are instantiated in their corresponding set-theoretic universes, which relate in diverse ways as forcing extensions, large cardinal ultrapowers, definable inner models and so on \cite{hamkinstoappear}. Although the generic multiverse of a model of set theory is merely a local neighborhood within the broader multiverse, nevertheless we take it that our project in this article, to study the forcing modalities of the generic multiverse, surely engages with the multiverse perspective.

Due to the fact that the relations $\FE$ and $\FEbar$ are converses of each other, the bimodal logic of the Kripke structure  $(\mathrm{GM}(V),\FE,\FEbar)$ is similar to temporal logics where the modality $\mathbf{F}$ (``there is a time in the future'') is converse to the modality $\mathbf{P}$ (``there is a time in the past'') \cite{JvBbook}. It is our overall aim to investigate this bimodal logic of the general multiverse and find out which validities hold in general (provably in $\ZFC$) or in the multiverse generated from specific universes.

\subsection{The modal logic of forcing and related work}

In our previous paper \cite{HamkinsLoewe2008:TheModalLogicOfForcing}, we introduced the modal logic of forcing and proved that the \ZFC-provably valid principles of forcing were exactly those in the modal theory known as \theoryf{S4.2}. The modal logic of forcing corresponds to the monomodal fragment of the modal logic discussed in \S\,\ref{ssec:multiverse} that only uses the relation $\FE$; or to the part of the generic multiverse that is generated only by the operation of taking forcing extensions and not ground models.

In \cite{HamkinsLoewe2008:TheModalLogicOfForcing}, we not only consider the $\ZFC$-provable modal logic of forcing, but also the modal logic of forcing of particular universes $V$. We show that this modal logic always lies between $\theoryf{S4.2}$ and $\theoryf{S5}$ and that the two extreme values are realized (for more details, cf. \S\,\ref{ssec:definitions}).
Various other aspects of the modal logic of forcing are considered in
\cite{Leibman2004:Dissertation,HamkinsWoodin2005:NMPccc,Fuchs2008:ClosedMaximalityPrinciples,Fuchs2009:CombinedMaximalityPrinciplesUpToLargeCardinals,Leibman2010:TheConsistencyStrengthOfMPccc(R),Rittberg2010:TheModalLogicOfForcing,FriedmanFuchinoSakai:OnTheSetGenericMultiverse,EsakiaLoewe,HLLsubmitted}.
The paper \cite{icla2009} presented at ICLA 2009 gives an overview of the status of research and creates a connection between the modal logic of forcing and ``set-theoretic geology'', i.e., going down from a universe to its ground models. This connection is further developed in this paper.

\subsection{The results of this paper}

As stated above, our overall aim is to understand the bimodal logic of the generic multiverse. However, we know rather little about the general status of the bimodal logic (cf.\ Footnote \ref{fn:whatweknowaboutthemixedlogic}). Instead of the bimodal theory, we consider the two monomodal fragments (the modal logic of forcing and the modal logic of grounds) and their possible combinations.

The main result of this paper is Theorem~\ref{Theorem.S4.2Arises}, constructing a model of set theory whose modal logic of grounds is $\theoryf{S4.2}$. This theorem is the downward analogue of the main result in \cite{HamkinsLoewe2008:TheModalLogicOfForcing}, where we showed that the modal logic of forcing over $\bfL$ is $\theoryf{S4.2}$. Based on the proof idea of the main result, we then make a foray into the bimodal world, by considering which combinations are possible for the two monomodal fragments; i.e., for which pairs $(\Lambda,\Lambda^*)$ can we find a model $V$
such that $\MLF{V} = \Lambda$ and $\MLG{V}=\Lambda^*$. We consider all possible combinations with $\Lambda,\Lambda^* \in \{\theoryf{S4.2},\theoryf{S5}\}$.

The paper is organized as follows: in \S\,\ref{ssec:definitions}, we give the necessary definitions in order to discuss the basic properties of the modal logic of grounds (and how it differs from the modal logic of forcing) in \S\,\ref{sec:basic}. In \S\,\ref{sec:combinations}, we finally consider the various combinations of upward and downward modal logics. The paper is not self-contained and uses ideas and concepts from the papers
\cite{HamkinsLoewe2008:TheModalLogicOfForcing,Reitz2007:TheGroundAxiom}; the proofs of our theorems are sketches and will be presented in more detail in the journal version of the paper.

\subsection{Definitions}\label{ssec:definitions}

In the following, we denote by $\mathrm{P}$ a countable set of propositional letters; for modal operators $\Boxup$ and $\Boxdown$,
$\mathcal{L}_\subBoxup$ and $\mathcal{L}_{\subBoxup,\subBoxdown}$ are the monomodal and bimodal propositional languages with the appropriate operators. We assume that the reader is familiar with the standard axioms and systems of monomodal logic, in particular,
$\axiomf{.2}$, $\theoryf{S4}$, $\theoryf{S4.2}$, and $\theoryf{S5}$ (if not, there is a summary in \cite[\S\,1]{HamkinsLoewe2008:TheModalLogicOfForcing}).

By $\mathcal{L}_\in$, we denote the first-order language of set theory with its set of sentences $\mathrm{Sent}(\mathcal{L}_\in)$. Any function $I : \mathrm{P} \to \mathrm{Sent}(\mathcal{L}_\in)$ is called an \emph{interpretation}. An interpretation $I$ generates a valuation of any Kripke structure $(F,R)$ in which the worlds consist of models $M$ of set theory and $R$ is any relation between them: via
$$I^*(M) := \{ p \,;\, M\models I(p)\},$$
$(F,R,I^*)$ becomes a Kripke model (similarly, for the bimodal language if we have two relations on $F$).

Note that in our special case (when $R = {\FE}$), the validity of a modal formula at a world of the Kripke model is not just a meta-theoretic property of the Kripke model, but can be expressed in the language of set theory: e.g., $\Boxup\varphi$ is interpreted by  ``for all generic extensions, $\varphi$ holds'' which by the Forcing Theorem \cite[Theorem 14.6]{Jech:SetTheory3rdEdition} is equivalent to ``for all Boolean algebras $\mathbb{B}$, we have that $\llbracket\varphi\rrbracket_\mathbb{B} = \mathbf{1}$''. Similarly, if $R={\FEbar}$, we can use a theorem of Laver's (cf.\ \cite{Laver2007:CertainVeryLargeCardinalsNotCreated}) that the ground model is definable with parameters in the forcing extension, in order to see that ``for all ground models, $\varphi$ holds'' is expressible in the language of set theory (this observation is due to Reitz, cf.\ \cite[Theorem 8]{geology}).

These observations are closely related to Woodin's result about \emph{multiverse truth}:

\begin{thm}[Woodin, 2009] There is a recursive transformation $\varphi \mapsto \varphi^*$ such that $M\models \varphi^*$ is equivalent to the statement ``for every model $N$ in the generic multiverse generated by $M$, $\varphi$ is true''.
\end{thm}

The fact that the modalities are expressible in the language of set theory allows us to move from interpretations to \emph{translations}:
We call a function $H : \mathcal{L}_{\subBoxup,\subBoxdown} \to \mathrm{Sent}(\mathcal{L}_\in)$ a \emph{translation} if
\begin{itemize}
\item $H(\varphi\wedge\psi) = H(\varphi)\wedge H(\psi)$,
\item $H(\neg\varphi) = \neg H(\varphi)$,
\item $H(\Boxup\varphi)$ is the sentence stating ``for all forcing extensions $M$, we have $M\models H(\varphi)$'', and
\item $H(\Boxdown\varphi)$ is the sentence stating ``for all grounds $M$, we have $M\models H(\varphi)$''.
\end{itemize}

Now, we can define the modal logic of the multiverse and two of its fragments: fixing a universe $V$, we call

$$\mathrm{ML}(\Boxup,\Boxdown,V) := \{\varphi\in\mathcal{L}_{\subBoxup,\subBoxdown}\,;\,\mbox{ for all translations $H$, $V\models H(\varphi)$}\}$$
the \emph{modal logic of the generic multiverse} generated by $V$. Similarly,
\begin{align*}
\mathrm{ML}(\Boxup,V) &:= \{\varphi\in\mathcal{L}_\subBoxup\,;\,\mbox{ for all translations $H$, $V\models H(\varphi)\}$, and}\\
\mathrm{ML}(\Boxdown,V) &:= \{\varphi\in\mathcal{L}_\subBoxdown\,;\,\mbox{ for all translations $H$, $V\models H(\varphi)$}\}
\end{align*}
are the \emph{modal logic of forcing} and the \emph{modal logic of grounds} at $V$, respectively. Metaphorically, we think of forcing extensions going upwards and thus the relation being a ground model going downwards. We therefore use the words ``upward'' and ``downward'' to indicate which
modalities we are talking about: e.g., if we say that a model $V$ satisfies upward $\theoryf{S4.2}$, we mean that $\MLF{V} = \theoryf{S4.2}$; similarly, we talk of ``upward buttons'' and ``downward buttons'' (see below).

As mentioned, in \cite[Theorem 21]{HamkinsLoewe2008:TheModalLogicOfForcing}, we proved that for any universe $V$, we get that
$$\theoryf{S4.2} \subseteq \MLF{V} \subseteq \theoryf{S5}\mbox{,}$$
and that the two extreme values are obtained. This immediately implies that the $\ZFC$-provable modal logic of forcing
$$\{\varphi\in\mathcal{L}_\subBoxup\,;\,\mbox{ for all translations $H$, $\ZFC\vdash H(\varphi)$}\}$$
is exactly $\theoryf{S4.2}$. Two facts about forcing are crucial for this result: the first is that the axiom $\axiomf{.2}$ is always a validity for the modal logic of forcing over any universe $V$ (cf.\ \cite{Leibman2004:Dissertation} and \cite[Theorem 7]{HLLsubmitted} for the theoretical background), providing the lower bound; the second is the existence of independent switches over each model of set theory \cite[Theorem 17]{HamkinsLoewe2008:TheModalLogicOfForcing}. We shall see that the situation is quite different for the modal logic of grounds.
Note that it is not known whether the modal logic of forcing can obtain any other value than $\theoryf{S4.2}$ or $\theoryf{S5}$ \cite[Question 19]{HamkinsLoewe2008:TheModalLogicOfForcing}.

In order to show upper bounds for a modal logic, we used certain control statements called \emph{buttons} and \emph{switches}: a {\df switch} is a statement $\varphi$ such that $\varphi$ and $\neg\varphi$ are both necessarily possible. A {\df button} is a statement $\varphi$ that is necessarily possibly necessary. These controls are {\df independent} if they can be operated independently, without affecting the status of the others
(cf.\
\cite[p.\ 1789]{HamkinsLoewe2008:TheModalLogicOfForcing} or \cite[\S\,4]{HLLsubmitted} for more detail). We shall use the abstract results that produce upper bounds from the existence of control statements in the proof sketches in this paper. For this we call a function $\sigma : \mathrm{P} \to \mathcal{L}_\Box$ a \emph{substitution} (this is the purely modal version of our notion of interpretation above); every substitution induces a function $\hat\sigma$ on the entire set of modal formulas. If $(F,R,v,w)$ is a pointed Kripke model,
we let $\mathrm{ML}(F,R,v,w) := \{\varphi\in\mathcal{L}_\Box\,;$ for all substitutions $\sigma$, we have that $F,R,v,w\models \hat\sigma(\varphi)\}$.

\begin{thm}\label{thm:buttonswitch}
If a pointed reflexive and transitive Kripke model $(F,R,v,w)$ has arbitrarily large finite independent families of buttons and switches, then
$$\mathrm{ML}(F,R,v,w) \subseteq \theoryf{S4.2}$$
\cite[Theorem 13]{HLLsubmitted}.
\end{thm}

\begin{thm}\label{thm:s5}
If a pointed reflexive and transitive Kripke model $(F,R,v,w)$ has arbitrarily large finite independent families of switches, then
$$\mathrm{ML}(F,R,v,w) \subseteq \theoryf{S5}$$
\cite[Theorem 10]{HLLsubmitted}.
\end{thm}

\section{Basic Results about the modal logic of grounds.}
\label{sec:basic}

It is easy to see that every $\theoryf{S4}$ assertion is downward valid (since a ground of a ground is a ground), but things are not as easy with the axiom $\axiomf{.2}$. This axiom would be valid if the answer to the following question is ``Yes'':
\begin{qstn}\label{q:ddg}
Let $V$ be a model of set theory, and $M$ and $N$ two grounds of $V$, i.e., $V = M[G] = N[H]$ for some generic filters $G$ and $H$. Is there some model $K$ which is a ground of both $M$ and $N$, i.e., there are $K$-generic filters $G^*$ and $H^*$ such that $K[G^*] = M$ and $K[H^*] = N$?
In other words, is $\FEbar$
directed among the grounds of $V$?
\end{qstn}

However, we do not know the answer to this question. We shall say that a universe $V$ in whose generic multiverse the answer to Question~\ref{q:ddg} is ``Yes'' satisfies the axiom of \emph{downward directedness of grounds} ($\mathrm{DDG}$). In all universes for which we can determine the truth value of $\mathrm{DDG}$, it is true.

The upper bound for the modal logic of forcing was \theoryf{S5}, but the situation for the modal logic of grounds is different. We denote by
$\theoryf{PL}$ (for ``propositional logic'') the modal logic satisfying $\Box p \leftrightarrow \Diamond p \leftrightarrow p$, i.e., the modal logic of a single reflexive point. By $\GA$, we denote the \emph{ground axiom} of \cite{Reitz2007:TheGroundAxiom} stating that the universe is not a non-trivial forcing extension of an inner model.\footnote{The ground axiom is jointly due to Reitz and the first author; cf.\ \cite{Hamkins2005:TheGroundAxiom,Reitz2006:Dissertation}.}
Clearly, the constructible universe satisfies $\GA$.

\begin{obs} If $V\models \GA$, then $\MLF{V} = \theoryf{PL}$.\end{obs}

\section{$\theoryf{S4.2}$ as the modal logic of grounds}
\label{sec:maintheorem}

The following theorem provides us with a model in which we can determine the modal logic of grounds. Its proof will serve as the underlying idea for the results in \S\,\ref{sec:combinations}.

\begin{thm}\label{Theorem.S4.2Arises}
If \ZFC\ is consistent, then there is a model of \ZFC\ whose ground valid assertions are exactly those in the modal theory
\theoryf{S4.2}.
\end{thm}

\begin{proofsketch} The idea of this proof is to use the \emph{bottomless model} of \cite{Reitz2007:TheGroundAxiom}: Let $\mathrm{Reg}^\bfL$ denote the class of regular $\bfL$-cardinals. Force over $\bfL$ with
$$\P=\prod_{\gamma\in\mathrm{Reg}^\bfL}\Add(\gamma,1)\mbox{,}$$
where we use Easton support and we use $\Add(\gamma,1)$ as defined in $\bfL$. Let $V=\bfL[G]$, where $G$ is $\bfL$-generic for
this forcing.

First, following Reitz, we argue that every ground model of $V$ contains a tail $\bfL[G^\alpha]$, where $G^\alpha=G\restrict \P^\alpha$, and
$\P^\alpha=\P\restrict[\alpha,\infty)$. That is, $\P^\alpha$ is the tail forcing, using only the factors from $\alpha$ onward. Suppose that $W$ is a
ground of $V$, so that $W[H]=V=\bfL[G]$. If the filter $G^\alpha$ is not in $W$, then it has a name there, and the Boolean value of the statement that this name is decided in certain ways compatible with the actual values of $G^\alpha$ will be a strictly descending sequence of Boolean values in the $W$-forcing, which violates the chain condition of that forcing (when $\alpha$ is much larger than that forcing). So, for large $\alpha$, $G^\alpha$ is an element of $W$.

In \cite{Reitz2007:TheGroundAxiom}, Reitz used this to show that $V$ has no bedrock, and we use the same argument to show that for any two grounds $M$ and $N$, we find $\alpha$ and $\beta$ such that $\bfL[G^\alpha]$ is a ground of $M$ and $\bfL[G^\beta]$ is a ground of $N$. Then if $\mu := \max\{\alpha,\beta\}$, $\bfL[G^\mu]$ is a ground of both $M$ and $N$. This proves that $\mathrm{DDG}$ holds in $V$, and thus $\axiomf{.2}$.

We now show that there are no additional modal validities by using Theorem~\ref{thm:buttonswitch}: Divide the regular cardinals above $\aleph_\omega$ into $\omega$ many disjoint classes $\Gamma_n$, each containing unboundedly many cardinals. Enumerate each class $\Gamma_n = \{\gamma^n_\alpha\,;\,\alpha< \mathrm{Ord}\}$ in order. Let $s_n$ be the statement ``the least $\alpha$ such that there exists
an $\bfL$-generic subset of $\gamma^n_\alpha$ is even.'' These statements are all true in $V$, since the corresponding $\alpha$ is $0$ in every case, as
the forcing $G$ explicitly adds an $\bfL$-generic subset of every $\gamma^n_\alpha$, including $\alpha=0$. In any ground model $W$ of $V$, we can go to
a deeper ground which is a tail extension, and then selectively remove additional factors of $G$ on indices in each $\Gamma_n$ so as to realize any
desired configuration of the switches in $L$. So the $s_n$'s are independent switches. Now let $b_n$ be the statement: ``there is no $\bfL$-generic
subset of $\aleph_n^\bfL$''. This statement is false in $V$, but true in any ground model of $V$ omitting the factor at $\aleph_n^\bfL$. Furthermore, once
true, the statement remains true in any deeper ground. Thus, each $b_n$ is a button. Finally, all these buttons and switches are independent, because
each is controlled by removing disjoint factors of $G$.\footnote{Note that $\bfL$-generic Cohen subsets of
different regular cardinals in $\bfL$ are necessarily mutually generic.}
Thus, Theorem~\ref{thm:buttonswitch} yields that $\MLG{V} = \theoryf{S4.2}$.
\end{proofsketch}

It follows that the \ZFC-provably valid principles of the modal of logic of grounds is a theory containing \theoryf{S4} and contained within \theoryf{S4.2}. If $\mathrm{DDG}$ is a theorem of $\ZFC$, the $\ZFC$-provable modal logic of grounds is exactly $\theoryf{S4.2}$.

\section{Combinations}\label{sec:combinations}

So far, we have looked at the modal logic of forcing and the modal logic of grounds in isolation, but ultimately, we are interested in determining the entire bimodal logic of the multiverse. Currently, we know almost nothing about the validity of mixed bimodal formulas beyond those validities that follow from the fact that the modal operators $\Boxup$ and $\Boxdown$ are defined by converse relations.\footnote{E.g., we know that $p\to\Boxup\neg\Boxdown\neg p$ and
$p\to\Boxdown\neg\Boxup\neg p$ hold. \label{fn:whatweknowaboutthemixedlogic}}
However, we can say something about possible combinations of modal logics of forcing with modal logics of grounds
(in the case of Theorems \ref{thm:s42s5} and \ref{conj:s5s42} under mild large cardinal assumptions).

\begin{thm}\label{thm:doubles4.2}
If $\ZFC$ is consistent, then there is a model of $\ZFC$ whose modal logic of forcing and modal logic of grounds are both $\theoryf{S4.2}$.
\end{thm}

\begin{proofsketch} In fact, this is the model $V$ constructed in the proof of Theorem~\ref{Theorem.S4.2Arises}: since $\theoryf{S4.2}$ is a general lower bound for the modal logic of forcing, we only have to show that independent upward buttons and switches exist.
We have such a family for  $\bfL$,\footnote{Note that the buttons provided in the proof of
\cite[Lemma\,6.1]{HamkinsLoewe2008:TheModalLogicOfForcing} are problematic since we do not know how to prove their independence, but there are other independent buttons in that paper.  Cf.\ the discussion at the end of \cite[\S\,4]{HLLsubmitted}.} and by observing that the forcing to add $G$ was cardinal-preserving and the $\GCH$ holds, we can use the buttons proposed by Rittberg \cite{Rittberg2010:TheModalLogicOfForcing} or Friedman, Fuchino and Sakai \cite{FriedmanFuchinoSakai:OnTheSetGenericMultiverse} or, alternatively, our stationary buttons from \cite[Theorem 29]{HamkinsLoewe2008:TheModalLogicOfForcing} (provided that we start the other forcing above $\omega_1$). The switches are $\GCH$ at
$\aleph_{\omega+n}$.\end{proofsketch}

\begin{thm}\label{thm:s42s5}
If $\ZFC +$``there is an inaccessible cardinal $\delta$ in $\bfL$ such that $\bfL_\delta\elesub\bfL$'' is consistent, then there is a model of set theory whose modal logic of forcing is $\theoryf{S4.2}$ and whose modal logic of grounds is $\theoryf{S5}$.
\end{thm}

\begin{proofsketch} This proof is a combination of the construction in Theorem~\ref{Theorem.S4.2Arises} and an idea from \cite[Theorem 5]{Hamkins2003:MaximalityPrinciple}. We start with an inaccessible cardinal $\delta$ in $\bfL$ such that $\bfL_\delta\elesub \bfL$ and force as in the proof of Theorem~\ref{Theorem.S4.2Arises} 
with the Easton support product $\P := \prod_{\gamma\in\mathrm{Reg}^\bfL} \Add(\gamma,1)$
to obtain $\bfL[G]$.  Since $\delta$ was inaccessible in $\bfL$,  $\P_\delta$ is just $\P$ as defined in $\bfL_\delta$ and we took a direct limit at $\delta$; thus, we still have $\bfL_\delta[G_\delta]\elesub \bfL[G]$. We claim that $V := \bfL[G^\delta]$ satisfies the conclusion of the theorem.

Upward $\theoryf{S4.2}$ follows as in the proof of Theorem~\ref{thm:doubles4.2}. Let us show that all downward buttons are pushed (this will establish $\theoryf{S5}$ as a lower bound for the modal logic of grounds):
If there is a ground pushing a downward button, then this fact is expressible, and so there must be a ground of $\bfL_\delta[G_\delta]$ pushing that button, and so this same forcing works with $\bfL[G]$. So there is a
small forcing pushing that button. And this small forcing ground will contain $\bfL[G^\delta]$, so it is already pushed in $\bfL[G^\delta]$.

In order to show that we have exactly $\theoryf{S5}$ as the modal logic of grounds, we use Theorem~\ref{thm:s5} and observe that the switches of the proof of Theorem~\ref{Theorem.S4.2Arises} still work.
\end{proofsketch}

We remark that if a universe has an independent family of upward switches and buttons, then so does any ground of that universe. This means that upward \theoryf{S4.2} is downwards necessary. Similarly, if a model has downward buttons and switches, then this remains upward necessary.\footnote{Unfortunately, we cannot conclude that downward \theoryf{S4.2} is upwards necessary, since we do not know whether downward \axiomf{.2} is valid in every model of set theory.}

Now, we consider the dual situation to that of Theorem~\ref{thm:s42s5}: upward $\theoryf{S5}$ and downward $\theoryf{S4.2}$.

\begin{thm}\label{conj:s5s42}
If $\ZFC +$``there is an inaccessible cardinal $\delta$ in $\bfL$ such that $\bfL_\delta\elesub\bfL$'' is consistent, then there is a model of set theory whose modal logic of forcing is $\theoryf{S5}$ and whose modal logic of grounds is $\theoryf{S4.2}$.
\end{thm}

\begin{proofsketch}
Again, this proof is a combination of the constructions in Theorem~\ref{Theorem.S4.2Arises} and 
\cite[Theorem 5]{Hamkins2003:MaximalityPrinciple}. Start in $\bfL$ with $\bfL_\delta \elesub \bfL$ and $\delta$ inaccessible in $\bfL$. Force to $\bfL[G]$ with the Easton support product $\P := \prod_{\gamma\in\mathrm{Reg}^\bfL} \Add(\gamma,1)$, as in Theorem \ref{Theorem.S4.2Arises}. As before in the proof of Theorem \ref{thm:s42s5}, we still have $\bfL_\delta[G_\delta]\elesub \bfL[G]$ (as $\delta$ was inaccessible in $\bfL$). We can now appeal to the definability of forcing relation and the fact that $\bfL_\delta \elesub \bfL$: anything a condition $p$ forces over $\bfL$ with $\P$ is the same as what it forces over $\bfL_\delta$ with $\P_\delta$.

Finally, we perform the forcing to obtain upwards $\theoryf{S5}$ from \cite[Theorem 5]{Hamkins2003:MaximalityPrinciple} to obtain $\bfL[G][h]$. This forcing is the same as the L\'{e}vy
collapse making $\delta$ into $\omega_1$. In $\bfL[G][h]$, we have upward $\theoryf{S5}$. However, we also have downward $\axiomf{.2}$: any ground of $\bfL[G][h]$ will contain some tail $\bfL[G^\alpha]$ just as in Theorem \ref{Theorem.S4.2Arises}, and thus we can take maximums as there to verify $\axiomf{.2}$. But now we also have downward buttons and
switches, just as in Theorem \ref{Theorem.S4.2Arises}.
\end{proofsketch}

Theorems~\ref{thm:doubles4.2}, \ref{thm:s42s5}, and \ref{conj:s5s42} take care of three of the four possible distributions of the theories $\theoryf{S4.2}$ and $\theoryf{S5}$ to the upward and downward modal logics. This naturally raises the question whether it is possible to have $\theoryf{S5}$ in both directions. We'll close this paper with the simple proof of the negative answer to this question:

\begin{thm}
There is no model of set theory such that both its modal logic of forcing and its modal logic of grounds are $\theoryf{S5}$.
\end{thm}

\begin{proof} Let us call an ordinal $\alpha$ a \emph{ground cardinal} if there is a ground in which $\alpha$ is a cardinal. Let $\gamma$ be the least infinite ground cardinal; clearly, $\gamma \leq \omega_1$. The statement ``$\gamma = \omega_1$'' is a downward button, but its negation is an upward button. So, if we had a model $V$ with $\mathrm{ML}(\Boxup,V) = \theoryf{S5} = \mathrm{ML}(\Boxdown,V)$, then $\gamma = \omega_1$ would have to be true by downward $\theoryf{S5}$, but false by upward $\theoryf{S5}$. \end{proof}

\bibliographystyle{abbrv}
\bibliography{icla2013}

\begin{thebibliography}{10}

\bibitem{EsakiaLoewe}
L.~Esakia and B.~L\"owe.
\newblock Fatal {H}eyting algebras and forcing persistent sentences.
\newblock {\em Studia Logica}, 100(1-2):163--173, 2012.

\bibitem{FriedmanFuchinoSakai:OnTheSetGenericMultiverse}
S.~Friedman, S.~Fuchino, and H.~Sakai.
\newblock On the set-generic multiverse, 2012.
\newblock submitted.

\bibitem{Fuchs2008:ClosedMaximalityPrinciples}
G.~Fuchs.
\newblock Closed maximality principles: implications, separations and
  combinations.
\newblock {\em Journal of Symbolic Logic}, 73(1):276--308, 2008.

\bibitem{Fuchs2009:CombinedMaximalityPrinciplesUpToLargeCardinals}
G.~Fuchs.
\newblock Combined maximality principles up to large cardinals.
\newblock {\em Journal of Symbolic Logic}, 74(3):1015--1046, 2009.

\bibitem{geology}
G.~Fuchs, J.~D. Hamkins, and J.~Reitz.
\newblock Set-theoretic geology, 2011.
\newblock submitted.

\bibitem{Hamkins2003:MaximalityPrinciple}
J.~D. Hamkins.
\newblock A simple maximality principle.
\newblock {\em Journal of Symbolic Logic}, 68(2):527--550, 2003.

\bibitem{Hamkins2005:TheGroundAxiom}
J.~D. Hamkins.
\newblock The {Ground Axiom}.
\newblock {\em Oberwolfach Reports}, 2(4):3160--3162, 2005.

\bibitem{icla2009}
J.~D. Hamkins.
\newblock Some second order set theory.
\newblock In R.~Ramanujam and S.~Sarukkai, editors, {\em Logic and its
  applications. Third Indian Conference, ICLA 2009, Chennai, India, January
  7-11, 2009. Proceedings}, volume 5378 of {\em Lecture Notes in Computer
  Science}, pages 36--50. Springer, Berlin, 2009.

\bibitem{hamkinstoappear}
J.~D. Hamkins.
\newblock The set-theoretical multiverse.
\newblock {\em Review of Symbolic Logic}, to appear.

\bibitem{HLLsubmitted}
J.~D. Hamkins, G.~Leibman, and B.~L\"owe.
\newblock Structural connections between a forcing class and its modal logic,
  2012.
\newblock submitted.

\bibitem{HamkinsLoewe2008:TheModalLogicOfForcing}
J.~D. Hamkins and B.~L{\"o}we.
\newblock The modal logic of forcing.
\newblock {\em Transactions of the American Mathematical Society},
  360(4):1793--1817, 2008.

\bibitem{HamkinsWoodin2005:NMPccc}
J.~D. Hamkins and W.~H. Woodin.
\newblock The necessary maximality principle for c.c.c.\ forcing is
  equiconsistent with a weakly compact cardinal.
\newblock {\em Mathematical Logic Quarterly}, 51(5):493--498, 2005.

\bibitem{Jech:SetTheory3rdEdition}
T.~Jech.
\newblock {\em Set Theory}.
\newblock Springer Monographs in Mathematics. Springer-Verlag, Heidelberg, 3rd
  edition, 2003.

\bibitem{Laver2007:CertainVeryLargeCardinalsNotCreated}
R.~Laver.
\newblock Certain very large cardinals are not created in small forcing
  extensions.
\newblock {\em Annals of Pure and Applied Logic}, 149(1--3):1--6, 2007.

\bibitem{Leibman2004:Dissertation}
G.~Leibman.
\newblock {\em Consistency strengths of modified maximality principles}.
\newblock PhD thesis, City University of New York, 2004.

\bibitem{Leibman2010:TheConsistencyStrengthOfMPccc(R)}
G.~Leibman.
\newblock The consistency strength of {${\rm MP}_{\rm CCC}(\Bbb R)$}.
\newblock {\em Notre Dame Journal of Formal Logic}, 51(2):181--193, 2010.

\bibitem{Reitz2006:Dissertation}
J.~Reitz.
\newblock {\em The {Ground Axiom}}.
\newblock PhD thesis, City University of New York, September 2006.

\bibitem{Reitz2007:TheGroundAxiom}
J.~Reitz.
\newblock The {Ground Axiom}.
\newblock {\em Journal of Symbolic Logic}, 72(4):1299--1317, 2007.

\bibitem{Rittberg2010:TheModalLogicOfForcing}
C.~J. Rittberg.
\newblock The modal logic of forcing.
\newblock Master's thesis, Westf\"alische Wilhelms-Universit\"at M\"unster,
  2010.

\bibitem{JvBbook}
J.~F. A.~K. van Benthem.
\newblock {\em The logic of time. A model-theoretic investigation into the
  varieties of temporal ontology and temporal discourse}, volume 156 of {\em
  Synthese Library}.
\newblock D. Reidel Publishing Co., Dordrecht, 1983.

\bibitem{woodinpreprint}
W.~H. Woodin.
\newblock The continuum hypothesis, the generic multiverse of sets, and the
  $\omega$ conjecture, 2009.
\newblock preprint.

\end{thebibliography}

\end{document}